\documentclass{article}
\usepackage{mathrsfs}
\usepackage{amssymb}
\usepackage{amsthm}
\usepackage{amsmath}
\usepackage{tikz-cd}
\usepackage{caption}
\usepackage{fancyhdr}
\usepackage{pgfplots}
\usepackage[shortlabels]{enumitem}
\usepackage{mathtools}
\usepackage{yfonts}
\usepackage{relsize}
\usepackage{xcolor}
\usepackage[nottoc,notlot,notlof]{tocbibind}
\usepackage{graphicx}
\usepackage{array}
\usepackage{scalerel,stackengine}
\stackMath
\usepackage{lipsum,hyperref}

\hypersetup{
     colorlinks   = true,
     linkcolor    = blue,
     citecolor    = red
}

\usetikzlibrary{calc}
\usetikzlibrary{arrows}
\usetikzlibrary{decorations.pathreplacing}
\usetikzlibrary{decorations.markings}

\theoremstyle{plain}
\newtheorem{theorem}{Theorem}
\numberwithin{theorem}{section}
\newtheorem*{theorem*}{Theorem}
\newtheorem{definition}[theorem]{Definition}
\newtheorem*{definition*}{Definition}

\newtheorem*{Proposition*}{Proposition}
\newtheorem*{criterion*}{Criterion}
\newtheorem{remark}[theorem]{Remark}

\newtheorem*{example*}{Example}
\newtheorem*{remark*}{Remark}
\newtheorem*{lemma*}{Lemma}

\newcommand{\Z}{\mathbb{Z}}

\newcommand{\C}{\mathbb{C}}
\newcommand{\R}{\mathbb{R}}
\newcommand{\calf}{\mathcal{F}}
\newcommand{\calo}{\mathcal{O}}
\newcommand{\changefont}{%
    \fontsize{10}{11}\selectfont
}

\title{Cartan's method and its applications in sheaf cohomology}
\author{Yuan Liu}
\date{}

\fancypagestyle{specialfooter}{%
  \fancyhf{}
  \fancyhead[]{}

  \fancyfoot[L]{
  \vspace{2pt}
  \changefont \textit{Date}: \today
  }
}

\begin{document}

\maketitle
\thispagestyle{specialfooter}
\begin{abstract}
This paper aims to use Cartan's original method in proving Theorem A and B on closed cubes to provide a different proof of the vanishing of sheaf cohomology over a closed cube if either (i) the degree exceeds its real dimension or (ii) the sheaf is (locally) constant and the degree is positive. In the first case, we can further use Godement's argument to show the topological dimension of a paracompact topological manifold is less than or equal to its real dimension. 
\end{abstract}

\section{Introduction}
Going back to the ``Séminaire Henri Cartan 1951--1952" (\cite{cartan}), H. Cartan proved Theorem A and B for coherent sheaves on Stein manifolds, for which he started with proving Theorem A and B for the sheaf of holomorphic functions on closed cubes in $\C^n$. Now we want to show that Cartan's original method works equally well as the standard arguments using C\v{e}ch cohomology in proving some well-known results on the vanishing of sheaf cohomology on closed cubes.
\\
\\
This partially expository paper is organized as follows. In section \ref{cartan's_method}, we go over Cartan's method in proving theorem B for the sheaf of holomorphic functions on closed cubes. In section \ref{application_of_cartan_method}, we provide the applications of Cartan's method in the vanishing of sheaf cohomology on closed cubes. The first result indicates that the cohomology vanishes for a sheaf of abelian groups on a closed cube if the degree exceeds the real dimension of this cube. Moreover, we can use Godement's argument to show that a paracompact topological manifold of real dimension $n$ has its topological dimension less than or equal to $n$. The second result shows that the cohomology of a constant sheaf on closed cubes vanishes if the degree is positive.\\
\\
\textbf{Acknowledge:} the author would like to express his sincere gratitude to Professor Mohan Ramachandran for his inspiration and guidance.
\section{Cartan's method}\label{cartan's_method}
In this section, we elaborate on the idea used by H. Cartan in the proof of Theorem B for the sheaf of holomorphic functions on closed cubes. For essential knowledge of sheaf cohomology, we refer the readers to the monograph of Godement \cite{Godement}.\\
\\
We first define closed cubes.
\begin{definition}
A closed cube of real dimension $n$ is a closed subset of $\R^n$ which can be written as
$$I^n=\prod\limits_{i=1}^n [a_i,b_i]=\{x\in \R^n:\ a_i\leqslant x_i\leqslant b_i, \text{\ where\ } 1\leqslant i\leqslant n\}$$
Here $a_i<b_i$ are finite real numbers\footnote{Notice that the interior is always assumed to be nonempty for these closed cubes.} for any $1\leqslant i\leqslant n$.
\end{definition}
For a closed cube of real dimension $r$ ($0\leqslant r\leqslant 2n$) in $\C^n$, we simply mean a closed cube $I^{r}=\prod\limits_{i=1}^{r}[a_i,b_i]$ with the remaining $2n-r$ coordinates fixed. The complex coordinate is $\{z_i=x_{2i-1}+\sqrt{-1}x_{2i} \text{,\ where\ } 1\leqslant i\leqslant n\}$.\\
\\
Now we are ready to state Cartan's theorem B for the sheaf of holomorphic functions on closed cubes in $\C^n$ and to provide his original proof.
\begin{theorem}\label{cartan_b_cube}
Suppose $X$ is a closed cube of real dimension $r$ in $\C^n$, i.e. $X=I^{r}=\prod\limits_{i=1}^{r}[a_i,b_i]$ with $0\leqslant r\leqslant 2n$, and $\calo$ is the sheaf of holomorphic functions on $\C^n$, then 
\begin{equation}\label{cartan_b_o}
H^p(X,\calo_{\C^n}\big|_{X})=0
\end{equation}
for $p>0$.
\end{theorem}
\begin{proof}
The proof uses induction on the real dimension $r$ of the closed cube.\\
Applying Godement's canonical flabby resolution (c.f. section 4.3, \cite{Godement}) to $\calo$, we get the long exact sequence
$$0\to\calo\to \mathcal{C}^0(\calo)\to\cdots\to\mathcal{C}^{p-1}(\calo)\to \mathcal{C}^{p}(\calo) \to \mathcal{C}^{p+1}(\calo)\to\cdots$$
where $\mathcal{C}^i(\calo)$'s, $i\geqslant 0$, are all canonical flabby sheaves. We also have the corresponding global sections over $I^r=\prod\limits_{i=1}^{r}[a_i,b_i]$:
\begin{equation}
\begin{split}
0&\to\calo(I^r)\xrightarrow{\imath} \mathcal{C}^0(\calo)(I^r)\xrightarrow{d^0} \mathcal{C}^{1}(\calo)(I^r)\to\cdots\\
\cdots&\to\mathcal{C}^{p-1}(\calo)(I^r)\xrightarrow{d^{p-1}} \mathcal{C}^{p}(\calo)(I^r) \xrightarrow{d^p} \mathcal{C}^{p+1}(\calo)(I^r)\to\cdots
\end{split}
\end{equation}
For $r=0$, it is trivially true because     $\mathcal{C}^i(\calo)=\calo$ and $d^i=\text{identity}$ for all $i$, we have
$H^p(I^0,\calo)=H^p\big(\mathcal{C}^{\bullet}(\calo)(I^0)\big)=0$ for $p>0$.\\
Now we assume it is true for closed cubes with real dimension $r-1 (1\leqslant r\leqslant 2n)$ and we want to prove it for $r$.\\
Define $\pi=\pi_1$ as the natural projection onto the first coordinate:
\begin{equation*}
I^{r}=[a_1,b_1]\times\cdots\times[a_r,b_r]\xrightarrow{\pi} I_1=[a_{1},b_{1}].
\end{equation*}
By induction hypothesis, for any $t\in [a_{1},b_{1}]$, the fibre $\pi^{-1}(t)$ satisfies
$$H^p\big(\pi^{-1}(t),\calo\big|_{\pi^{-1}(t)}\big)=0.$$
\textbf{Case 1:} if $p\geqslant 2$.\\
For any element of $H^p(X,\calo\big|_{I^r})$, take its representation $\eta\in\mathcal{C}^p(\calo)(I^r)$ with $d^p\eta=0$, then 
$$\eta\big|_{\pi^{-1}(t)}=d^{p-1}\alpha_t$$
for some $\alpha_t\in\mathcal{C}^{p-1}\big(\pi^{-1}(t),\calo\big|_{\pi^{-1}(t)}\big)$.\\
By Th\'eor\`eme 4.11.1 (Page 193) of  \cite{Godement}, we have $$H^p(\pi^{-1}(t),\calo\big|_{\pi^{-1}(t)})=\varinjlim\limits_{\substack{\pi^{-1}(t)\subseteq U\\U \text{\ is open in\ }\R^r}}H^p(U,\calo\big|_U)$$
Hence there is a neighborhood $U_t\subset \R^r$ of $\pi^{-1}(t)$ such that
$$\eta\big|_{U_t}=d^{p-1}\alpha_t\big|_{U_t}$$
for some $ \alpha_t\in\mathcal{C}^{p-1}(\calo)(U_t)$.\\
By the flabbiness of $\mathcal{C}^{p-1}(\calo)$, we can assume that $\alpha_t\in\mathcal{C}^{p-1}(\calo)(I^r)$.\\
Now for such $U_t$'s covering $I_1=[a_1,b_1]$, we can find a finite sub-covering. Each $U_t$ contains a closed cube of the form $\pi^{-1}([t-\varepsilon,t+\varepsilon])$ for some $\varepsilon>0$, and we replace $U_t$ by such closed cubes.
\tikzset{
  on each segment/.style={
    decorate,
    decoration={
      show path construction,
      moveto code={},
      lineto code={
        \path [#1]
        (\tikzinputsegmentfirst) -- (\tikzinputsegmentlast);
      },
      curveto code={
        \path [#1] (\tikzinputsegmentfirst)
        .. controls
        (\tikzinputsegmentsupporta) and (\tikzinputsegmentsupportb)
        ..
        (\tikzinputsegmentlast);
      },
      closepath code={
        \path [#1]
        (\tikzinputsegmentfirst) -- (\tikzinputsegmentlast);
      },
    },
  },
  mid arrow/.style={postaction={decorate,decoration={
        markings,
        mark=at position .5 with {\arrow[#1]{stealth}}
      }}},
}
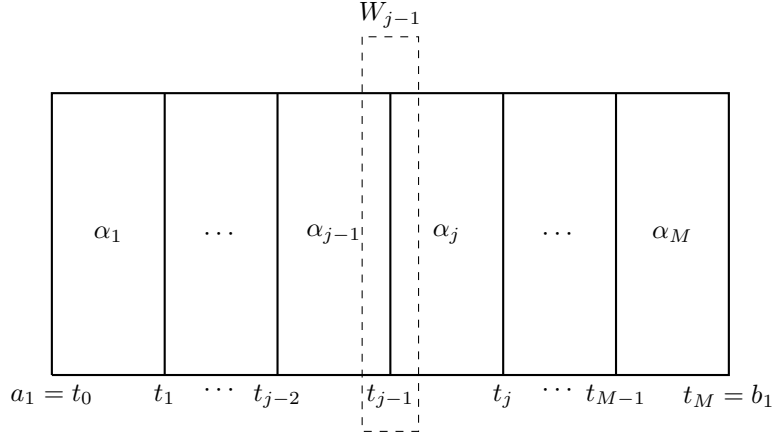
\begin{figure}[htp]
\centering
\begin{tikzpicture}[scale=0.75]{saf}
\draw[thick] (-6,-2.5)--(-6,2.5)--(6,2.5)--(6,-2.5)--(-6,-2.5);
\draw[thick] (0,-2.5)--(0,2.5);
\draw[thick] (-2,-2.5)--(-2,2.5);
\draw[thick] (2,-2.5)--(2,2.5);
\draw[thick] (4,-2.5)--(4,2.5);
\draw[thick] (-4,-2.5)--(-4,2.5);
\node[below] at (-4,-2.5) {$t_{1}$};
\node[below] at (-3,-2.5) {$\cdots$};
\node[below] at (-2,-2.5) {$t_{j-2}$};
\node[below] at (2,-2.5) {$t_{j}$};
\node[below] at (4,-2.5) {$t_{M-1}$};
\node[below] at (3,-2.5) {$\cdots$};
\node[below] at (0,-2.5) {$t_{j-1}$};
\node[below] at (6,-2.5) {$t_{M}=b_1$};
\node[below] at (-6,-2.5) {$a_1=t_0$};
\node at (-1,0) {$\alpha_{j-1}$};
\node at (1,0) {$\alpha_{j}$};
\node at (-5,0) {$\alpha_{1}$};
\node at (5,0) {$\alpha_{M}$};
\node at (-3,0) {$\cdots$};
\node at (3,0) {$\cdots$};
\draw[dashed] (-0.5,3.5)--(-0.5,-3.5)--(0.5,-3.5)--(0.5,3.5)--(-0.5,3.5);
\node[above] at (0,3.5) {$W_{j-1}$};
\end{tikzpicture}
\label{cartan_picture}
\caption{Cartan's method}
\end{figure}
\\
Now let $a_1=t_0<t_1<\cdots<t_j<\cdots<t_{M-1}<t_M=b_1$ be a partition such that we have $\alpha_j\in \mathcal{C}^{p-1}(\calo)\big(\pi^{-1}([t_{j-1},t_j])\big)$ with
$$\eta\big|_{\pi^{-1}([t_{j-1},t_j])}=d^{p-1}\Big(\alpha_j\big|_{\pi^{-1}([t_{j-1},t_j])}\Big)$$
for each $1\leqslant j\leqslant M$.\\
On $\pi^{-1}(t_{j-1})$, we have
$$d^{p-1}\Big((\alpha_j-\alpha_{j-1})\big|_{\pi^{-1}(t_{j-1})}\Big)=\eta\big|_{\pi^{-1}(t_{j-1})}-\eta\big|_{\pi^{-1}(t_{j-1})}=0.$$
Therefore 
$$\alpha_j-\alpha_{j-1}=d^{p-2}\beta_{j-1}$$
for some $\beta_{j-1}\in\mathcal{C}^{p-2}(\calo)(W_{j-1})$ in a neighborhood $W_{j-1}\subset \R^r$ of $\pi^{-1}(t_{j-1})$ due to the induction hypothesis and $p-1 > 0$. By flabbiness of $\mathcal{C}^{p-2}(\calo)$ (since $p-2\geqslant 0$), we can assume that $\beta_{j-1}\in \mathcal{C}^{p-2}(\calo)(I^r)$ .\\
Define
$$\alpha'=\begin{cases} 
      \alpha_{j-1}+d^{p-2}\beta_{j-1} & \text{\ on\ }\pi^{-1}([t_{j-2},t_{j-1}]) \\
      \alpha_{i} & \text{\ on\ }\pi^{-1}([t_{j-1},t_j])
   \end{cases}$$
then we get $\alpha'\in\mathcal{C}^{p-1}(\calo)\big(\pi^{-1}([t_{j-2},t_j])\big)$ such that $d^{p-1}\alpha'=\eta\big|_{\pi^{-1}([t_{j-2},t_j])}.$
Follow this process and modify for $M-1$ times from left to right, we get $\alpha\in \mathcal{C}^{p-1}(\calo)(I^r)$ with $d^{p-1}\alpha=\eta$ on $I^r$.\\
This proves that $H^p(I^r,\calo_{\C^n}\big|_{I^r})=0$ for $p>1$.\\
\textbf{Case 2:} if $p=1$. \\
The first several terms of canonical flabby resolution is:
$$0\to \calo(I^r)\xrightarrow{\imath}\mathcal{C}^0(\calo)(I^r)\xrightarrow{d^0}\mathcal{C}^1(\calo)(I^r)\xrightarrow{d^1}\mathcal{C}^2(\calo)(I^r)\to\cdots$$
For any element in $H^1(I^r,\calo\big|_{I^r})$, take its representation $\eta\in\calf^1(I^r)$ with $d^1\eta=0$. By the same argument as in case 1, we can find a partition of $[a_1,b_1]$, say $a_1=t_0<t_1<\cdots<t_j<\cdots<t_{M-1}<t_{M}=b_1$, such that
$$\eta\big|_{\pi^{-1}([t_{j-1},t_j])}=d^0\alpha_{j}$$
for some $\alpha_{j}\in \mathcal{C}^0(\calo)(\pi^{-1}([t_{j-1},t_j]))$. By flabbiness of $\mathcal{C}^0(\calo)$, we can assume that $\alpha_{j}\in\mathcal{C}^0(\calo)(I^r)$ .\\
Then we get
$$\eta\big|_{\pi^{-1}([t_{j-1},t_j])}=d^0\alpha_{j}\big|_{\pi^{-1}([t_{j-1},t_j])}.$$
On $\pi^{-1}(t_{j-1})$, we have
$$d^0\big(\alpha_{j}\big|_{\pi^{-1}(t_{j-1})}-\alpha_{j-1}\big|_{\pi^{-1}(t_{j-1})}\big)=\eta\big|_{\pi^{-1}(t_{j-1})}-\eta\big|_{\pi^{-1}(t_{j-1})}=0$$
thus $\alpha_{j}-\alpha_{j-1}\in\calo\big(\pi^{-1}(t_{j-1})\big)$ is defined on a closed neighborhood $W_{j-1}$ in $\R^r$ of $\pi^{-1}(t_{j-1})$ (here we identify $\calo $ with its image by $\imath$).\\
Now write the holomorphic function as $f=\alpha_{j}-\alpha_{j-1}$. Because $\pi^{-1}(t_{j-1})$ is compact, we can assume that $f$ is defined on the closed neighborhood $W_{j-1}=[a,b]\times\Big(\prod\limits_{i=2}^{r}[a_i-\varepsilon,b_i+\varepsilon]\Big)$\footnote{ When $r=1$, we take $W_{j-1}$ as a uniform small neighborhood of the form $[a,b]\times[-\varepsilon,\varepsilon]$ near each $t_{j-1}$, so are the $B$ and $C_j$'s defined below.}, where $\varepsilon>0$ is a small and $t_{j-2}<a<t_{j-1}<b<t_{j}$.\\
We need the following notations.
\begin{itemize}
\item $B=[a,b]\times[a_{2}-\varepsilon,b_{2}+\varepsilon]$ is a rectangle in $z_1$-plane where the function $f$ is defined and is holomorphic.
\item Let $B_{j-1}=[t_{j-2},t_{j-1}]$ and $B_{j}=[t_{j-1},t_{j}]$. Also, denote $C_{j-1}=[t_{j-2},t_{j-1}]\times[a_{2},b_{2}]$ and $C_{j}=[t_{j-1},t_{j}]\times[a_{2},b_{2}]$ to be two closed rectangles in $z_1$-plane on the left and on the right respectively.
\item Write the boundary of $B$ as the union of $\gamma_1=\partial B\cap \{x_{1}<t_{j-1}\}$ and $\gamma_2=\partial B\cap \{x_{1}>t_{j-1}\}$. 
\end{itemize}
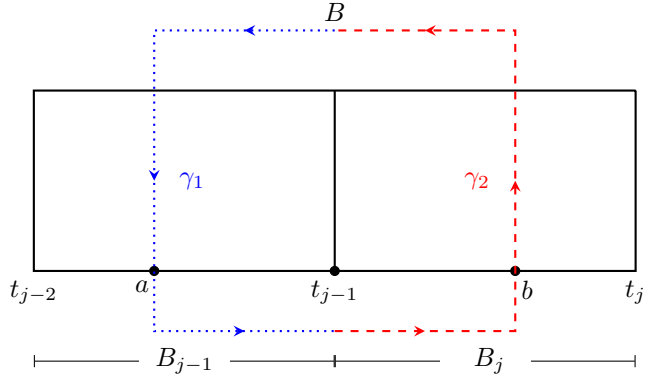
\begin{figure}[htp]
\centering
\begin{tikzpicture}[scale=0.8]
\draw[thick] (5,3)--(-5,3)--(-5,0)--(5,0)--(5,3);
\draw[thick] (0,0)--(0,3);
\node[left,below] at (-3.2,0) {$a$};
\node[below] at (0,0) {$t_{j-1}$};
\node[below] at (5,0) {$t_{j}$};
\node[below] at (-5,0) {$t_{j-2}$};
\node[below] at (3.2,0) {$b$};
\draw[fill] (0,0) circle (0.5ex);
\draw[fill] (3,0) circle (0.5ex);
\draw[fill] (-3,0) circle (0.5ex);
\path [draw=blue,thick,dotted, postaction={on each segment={mid arrow=blue}}]
  (0,4)--(-3,4)--(-3,-1)--(0,-1);
\path [draw=red,thick,dashed, postaction={on each segment={mid arrow=red}}]
  (0,-1)--(3,-1)--(3,4)--(0,4);
\draw[|-|] ($(-5,-1.5)$) -- node[fill=white,inner sep=1mm,text width=1cm,midway] {$\ B_{j-1}$}($(0,-1.5)$); \draw[|-|] ($(5,-1.5)$) -- node[fill=white,inner sep=1mm,text width=1cm,midway] {$\quad B_j$}($(0,-1.5)$);
\node[above] at (0,4) {$B$};
\node[left,blue] at (-2,1.5) {$\gamma_1$};
\node[right,red] at (2,1.5) {$\gamma_2$};
\end{tikzpicture}
\caption{Picture in $z_1$-plane} \label{rectangular}
\end{figure}
By Cauchy's integral formula, we have 
\begin{align*}
f(z_1,\cdots,z_n)=&\frac{1}{2\pi\sqrt{-1}}\int_{\partial B}\frac{f(\zeta_1,z_2,\cdots,z_n)}{\zeta_1-z_1}d\zeta_1\\
=&\frac{1}{2\pi\sqrt{-1}}\int_{\gamma_1}\frac{f(\zeta_1,z_2,\cdots,z_n)}{\zeta_1-z_1}d\zeta_1+\frac{1}{2\pi\sqrt{-1}}\int_{\gamma_2}\frac{f(\zeta_1,z_2,\cdots,z_n)}{\zeta_1-z_1}d\zeta_1\\
=&: g_1+g_2
\end{align*}
Since $\gamma_1$ is disjoint from $C_j$, $g_1$ defines a holomorphic function on $C_j$. Similarly, $g_2$ defines a holomorphic function on $C_{j-1}$. Now we know the functions $g_1$ and $g_2$ are holomorphic on $\pi^{-1}([t_{j-1},t_{j}])$ and $\pi^{-1}([t_{j-2},t_{j-1}])$ respectively due to Hartogs's theorem on separate holomorphicity.\\
Let $h_{j-1}=g_2\in\calo(\pi^{-1}([t_{j-2},t_{j-1}]))$ and $h_{j}=-g_1\in\calo(\pi^{-1}([t_{j-1},t_{j}]))$, then
$$\alpha_j\big|_{\pi^{-1}(t_{j-1})}-\alpha_{j-1}\big|_{\pi^{-1}(t_{j-1})}=f=h_{j-1}-h_{j}$$
Thus
$$\big(\alpha_j+h_j\big)\big|_{\pi^{-1}(t_{j-1})}=\big(\alpha_{j-1}+h_{j-1}\big)\big|_{\pi^{-1}(t_{j-1})}.$$
They patch together to get $\alpha'\in\mathcal{C}^0(\calo)\big(\pi^{-1}[t_{j-2},t_{j}]\big)$ with $$d^0\alpha'=\eta\big|_{\pi^{-1}([t_{j-2},t_{j}])}.$$
Continue the above process for finitely many (i.e. $M-1$) times, we get the desired $\alpha\in\mathcal{C}^0(\calo)(I^r)$ with $d^0\alpha=\eta$ on $I^r$. This proves that $H^1(I^r,\calo_{\C^n}\big|_{I^r})=0$.\\
By induction on $r$, we know that $H^p(I^r,\mathcal{O}_{\C^n}\big|_{I^r})=0$ for any $p>0$ and $0\leqslant r\leqslant 2n$.
\end{proof}
By almost the same argument, we can show that
\begin{theorem}
Let $X$ be a closed cube in $\C^n$ and $\Omega^{p}(X)$ be the sheaf of holomorphic $p$-forms on $X$, then
\begin{equation}
H^{q}(X, \Omega^{p}(X))=0
\end{equation}
for any $q>0$.
\end{theorem}
\begin{proof}
For $q\geqslant2$, this is no essential adjustment needed. For $q=1$, we only need to repeat the corresponding process above for each holomorphic coefficient in the $p$-forms.
\end{proof}
\section{Applications of Cartan's method in sheaf cohomology}\label{application_of_cartan_method}
For Cartan's method above, we do not use any C\v{e}ch cohomology; in fact, we modify local data on common faces of closed cubes and glue them to get a global one. This motivates its application in sheaf cohomology, for the sheaves are compatible local data that can be glued together uniquely. Now we give two applications of Cartan's method in sheaf cohomology.
\subsection{Topological dimension of paracompact topological manifolds}\label{topological_dimension}
The following topological dimension (or cohomology dimension) was first given by Godement (c.f. section 4.13 \cite{Godement}).
\begin{definition}[Topological dimension]\label{dimension}
A nonempty topological space $X$ is said to have topological dimension $\leqslant n$ if 
$$H^p(X,\calf)=0$$
for any $p>n$ and any abelian sheaf $\calf$ over $X$.\\
The topological dimension of X is defined as the least value for which the above is true.
\end{definition}
Given a paracompact topological manifold $X$ of real dimension $n$ in the usual sense, we want to show that its topological dimension is less than or equal to $n$. By an argument of Godement (c.f. Th\'{e}or\`{e}me 4.14.1, \cite{Godement}), we only need to prove this locally in a small neighborhood (for example, closed cubes) of any point. The rest relies on the vanishing of sheaf cohomology of degree $p>n$ on closed cubes, and the standard way to achieve this is to compute its C\v{e}ch cohomology (c.f. section 5.12 of \cite{Godement}). Now, we can use Cartan's method instead.
\begin{theorem}\label{cubes_sheaf}
Suppose $\calf$ is a sheaf of abelian groups in a neighborhood of a closed cube $I^n$, then
\begin{equation}
H^p(I^n,\calf)=0    
\end{equation}
if $p>n$.
\end{theorem}
\begin{proof}
We prove this by induction on $n$.\\
If $n=0$, $I^0$ is a point. Now consider Godement's flabby resolution for $\calf$:
$$0\to\calf\to \mathcal{C}^0(\calf)\to\mathcal{C}^1(\calf)\to\cdots$$
where $\mathcal{C}^i(\calf)$'s, $i\geqslant 0$, are all canonical flabby sheaves.\\
Then $H^p(I^0,\calf)=H^p\big(\mathcal{C}^{\bullet}(\calf)(I^0)\big)=0$ for $p>n=0$.\\
Now we assume it is true for $I^{k}$ with $0\leqslant k<n$.\\
Consider the natural projection map $\pi=\pi_1$ onto the first coordinate:
\begin{equation*}
I^n=[a_1,b_1]\times\cdots\times[a_n,b_n]\xrightarrow{\pi} I_1=[a_1,b_1].
\end{equation*}
Then the fibre $\pi^{-1}(t)$, for any $t\in [a_1,b_1]$, satisfies
$$H^p\big(\pi^{-1}(t),\calf\big|_{\pi^{-1}(t)}\big)=0$$
by the induction hypothesis and $p>n-1$.\\
Now consider the flabby resolution:
$$0\to\calf\to \mathcal{C}^0(\calf)\to\cdots\to\mathcal{C}^{p-1}(\calf)\to \mathcal{C}^{p}(\calf) \to \mathcal{C}^{p+1}(\calf)\to\cdots$$
and the corresponding global sections
\begin{equation*}
\begin{split}
0&\to\calf(I^n)\xrightarrow{\imath} \mathcal{C}^0(\calf)(I^n)\xrightarrow{d^0} \mathcal{C}^{1}(\calf)(I^n)\to\cdots\\
\cdots&\to\mathcal{C}^{p-1}(\calf)(I^n)\xrightarrow{d^{p-1}} \mathcal{C}^{p}(\calf)(I^n) \xrightarrow{d^p} \mathcal{C}^{p+1}(\calf)(I^n)\to\cdots.    
\end{split}
\end{equation*}
Take $\eta\in\mathcal{C}^p(\calf)(I^n)$ with $d^p\eta=0$ and $p>n$, then 
$$\eta\big|_{\pi^{-1}(t)}=d^{p-1}\alpha_t$$
for some $\alpha_t\in\mathcal{C}^{p-1}\big(\pi^{-1}(t),\calf\big|_{\pi^{-1}(t)}\big)$.\\
As in $p\geqslant 2$ case of section \ref{cartan's_method}, we can get a partition of $[a_1,b_1]$ by closed intervals, say $a_1=t_0<t_1<\cdots<t_j<\cdots<t_{M-1}<t_M=b_1$ such that
$$\eta\big|_{\pi^{-1}([t_{j-1},t_j])}=d^{p-1}\alpha_{j}$$
for some $\alpha_{j}\in\mathcal{C}^{p-1}(\calf)\big(\pi^{-1}([t_{i-1},t_i])\big)$ and each $1\leqslant j\leqslant M$.\\
On $\pi^{-1}(t_{j-1})$, we have
$$d^{p-1}\Big((\alpha_j-\alpha_{j-1})\big|_{\pi^{-1}(t_{j-1})}\Big)=\eta\big|_{\pi^{-1}(t_{j-1})}-\eta\big|_{\pi^{-1}(t_{j-1})}=0.$$
Therefore 
$$\alpha_j-\alpha_{j-1}=d^{p-2}\beta_{j-1}$$
for some $\beta_{j-1}\in\mathcal{C}^{p-2}(\calf)$ on a neighborhood of $\pi^{-1}(t_{i-1})$ due to the induction hypothesis and $p-1>n-1$.\\
We can assume that $\beta_{j-1}\in \mathcal{C}^{p-2}(\calf)(I_n)$ since $\mathcal{C}^{p-2}(\calf)$ is flabby when $p-2\geqslant 0$.\\
Now we define
$$\alpha'=\begin{cases} 
      \alpha_{j-1}+d^{p-2}\beta_{j-1} & \text{\ on\ }\pi^{-1}([t_{j-2},t_{j-1}]) \\
      \alpha_{j} & \text{\ on\ }\pi^{-1}([t_{j-1},t_j])
   \end{cases}$$
then we get $\alpha'\in\mathcal{C}^{p-1}(\calf)\big(\pi^{-1}([t_{i-2},t_i])\big)$ such that $d^{p-1}\alpha'=\eta\big|_{\pi^{-1}([t_{j-2},t_j])}.$
Follow this process and modify for $M-1$ times, we get $\alpha\in \mathcal{C}^{p-1}(\calf)(I^n)$ with $d^{p-1}\alpha=\eta$ on $I^n$. This proves that $H^p(I^n,\calf)=0$ for $p>n$.\\
By induction, we know that $H^p(I^n,\mathcal{F})=0$ for any $n$ and $p>n$.
\end{proof}
\textbf{Godement's argument and topological dimension}\\
\\
The following results can be found in the book of Godement (c.f. chapter 4, \cite{Godement}), and we include them for the convenience of readers and rewrite to make them compatible with our argument.\\
\\
We have the following criterion using soft resolution on topological dimensions. 
\begin{theorem}
For a paracompact Hausdorff space, the following are equivalent:
\begin{itemize}
\item[(a)] The topological dimension of $X$ is $\leqslant n$. 
\item[(b)] The sheaf
$$\mathcal{Z}^n(\calf)=\ker\{\mathcal{C}^n(\calf)\to \mathcal{C}^{n+1}(\calf)\}$$
is soft for every sheaf $\calf$ over $X$.
\item[(c)] Every sheaf $\calf$ over $X$ admits a resolution of length $n$ of soft sheaves:
\begin{equation*}
0\to\calf\to \mathcal{L}^0\to\mathcal{L}^1\to\cdots\to \mathcal{L}^n \to 0.  
\end{equation*}
\end{itemize}
\end{theorem}
\begin{proof}
(b) $\implies$ (c).\\
Recall that flabbiness implies softness. Take $\mathcal{L}^q=\mathcal{C}^q(\calf)$ for $q<n$ and $\mathcal{L}^n=\mathcal{Z}^n(\calf)$.\\
(c) $\implies$ (a).\\
For every sheaf $\calf$, the de-Rham-Weil isomorphism implies that 
$$H^q(X,\calf)=H^q(\mathcal{L}^{\bullet}(X))=0$$
when $q>n$.\\
(a) $\implies$ (b).\\
Let $\calf$ be a sheaf on $X$ and $S\xhookrightarrow{i}X$ be a closed subset. Let $U=X\backslash S$ and $j:U\hookrightarrow X$.\\
Denote
\[  (j_{!}j^{-1}\calf)_x=\left\{
\begin{array}{ll}
      \calf_x & \quad x\in U \\
      0 &\quad  \text{otherwise} \\
\end{array} 
\right. 
\]
Here $j_{!}$ is the operator given by the trial extension across closed subset.\\
Then the operator $j_{!}j^{-1}$ is given by first pull back $\calf$ over $X$ via $j$, then extend trivially to $X$. It preserves exactness and softness and since flabby sheaves are soft on paracompact Hausdorff spaces, this operator applied to a flabby resolution gives a soft resolution. Thus we have
$$0\to j_{!}j^{-1}\calf\to j_{!}j^{-1}\mathcal{C}^0(\calf)\to j_{!}j^{-1}\mathcal{C}^1(\calf)\to \cdots$$
is a soft resolution.\\
Also
$$\ker\big(j_{!}j^{-1}\mathcal{C}^n(\calf)\to j_{!}j^{-1}\mathcal{C}^{n+1}(\calf)\big)=j_{!}j^{-1}\mathcal{Z}^n(\calf).$$
By the usual de-Rham argument, we have
$$H^1(X,j_{!}j^{-1}\mathcal{Z}^n(\calf))\cong H^{n+1}(X,j_{!}j^{-1}\calf)=0.$$
The long exact sequence associated to
$$0\to j_{!}j^{-1}\mathcal{Z}^n(\calf)\to\mathcal{Z}^n(\calf)\to i_{*}i^{-1}\mathcal{Z}^n(\calf)\to 0$$
gives
$$\mathcal{Z}^n(\calf)(X)\to \mathcal{Z}^n(\calf)(S)=H^0(X,i_{*}i^{-1}\mathcal{Z}^n(\calf))\to H^1(X,j_{!}j^{-1}\mathcal{Z}^n(\calf))=0.$$
This shows that $\mathcal{Z}^n(\calf)$ is soft.
\end{proof}
We also have easy consequences of the above criterion.
\begin{theorem}\label{local}
Let X be a paracompact Hausdorff topological space.
\begin{enumerate}
\item if $S\subset X$, then the topological dimension of $S$ $\leqslant$ the topological dimension of $X$, providing either $S$ is closed or $X$ is metrizable.
\item if every point of $X$ has a neighborhood of topological dimension $\leqslant n$, then topological dimension of $X$ $\leqslant n$.
\end{enumerate}
\end{theorem}
\begin{proof}
1. We only prove for $S$ is closed. For the other case, see Th\'{e}or\`{e}me 4.14.2 \cite{Godement}.\\
Take $i: S\hookrightarrow X$ and $\calf$ be a sheaf on $S$. Then $i_{*}\calf$ is a sheaf over $X$ and 
$$H^p(S,\calf)\cong H^p(X,i_*\calf).$$
2. By the above criterion, we have $\mathcal{Z}^n(\calf)$ is locally soft. Because softness is a local property, we have $\mathcal{Z}^n(\calf)(X)$ is soft, thus the topological dimension of $X$ $\leqslant n$.
\end{proof}
Now we can prove the following theorem.
\begin{theorem}
Let $X$ be a paracompact topological manifold of real dimension $n$, and $\calf$ be a sheaf of abelian groups defined on $X$, then
\begin{equation}
H^p(X,\calf)=0\quad \text{when\quad } p>n.
\end{equation}
In other words, the topological dimension of a paracompact topological manifold is less than or equal to its real dimension in the usual sense.
\end{theorem}
\begin{proof}
For every point of $X$, it has a closed cube (with nonempty interior) of dimension $n$ as a neighborhood, which has topological dimension $\leqslant n$ due to Theorem \ref{cubes_sheaf}. By Theorem \ref{local}.2, we know the topological dimension of $X$ is also $\leqslant n$, which is equivalent to the required vanishing property of sheaf cohomology.
\end{proof}
\begin{remark}
This argument goes through as long as we can find a closed cube neighborhood of dimension less than or equal to $n$ for any given point in a topological space $X$. The case when $X$ is a paracompact topological manifold of dimension $n$ is only a special case.
\end{remark}
\subsection{Vanishing of cohomology for locally constant sheaves on closed cubes}\label{constant_sheaf_application}
If the sheaf defined on a closed cube is (locally) constant, we can furthermore use Cartan's method to show that any sheaf cohomology of positive degree vanishes. Precisely, we will show:
\begin{theorem}
Let $\mathcal{A}$ be a locally constant sheaf of abelian groups $A$ defined in a neighborhood of a closed cube $I^n$, then
\begin{equation}
H^p(I^n,\mathcal{A})=0    
\end{equation}
for any $p > 0$.
\end{theorem}
\begin{proof}
We use induction on $n$.\\
Consider the following flabby resolution:
$$0\to\mathcal{A}\to \mathcal{C}^0(\mathcal{A})\to\mathcal{C}^1(\mathcal{A})\to\mathcal{C}^2(\mathcal{A})\to\cdots\to\mathcal{C}^p(\mathcal{A})\to\mathcal{C}^{p+1}(\mathcal{A})\to\cdots$$
and its corresponding sections on $I^n$:
\begin{equation*}
\begin{split}
0&\to\mathcal{A}(I^n)\xrightarrow{\imath} \mathcal{C}^0(\mathcal{A})(I^n)\xrightarrow{d^0} \mathcal{C}^{1}(\mathcal{A})(I^n)\to\cdots\\
\cdots&\to\mathcal{C}^{p-1}(\mathcal{A})(I^n)\xrightarrow{d^{p-1}} \mathcal{C}^{p}(\mathcal{A})(I^n) \xrightarrow{d^p} \mathcal{C}^{p+1}(\mathcal{A})(I^n)\to\cdots\end{split}
\end{equation*}
When $n=0$, $I^0$ is simply a point and the result is true as before.\\
Assume that it is true for $n-1$.\\
\textbf{Case 1:} if $p\geqslant 2$, take $\eta\in\mathcal{C}^p(\mathcal{A})(I^n)$ with $d^p \eta=0$. The proof is analogous to that of theorem \ref{cubes_sheaf}, which we will not repeat here. Also, we use the same notation as in section \ref{topological_dimension} (with $\pi=\pi_1$).\\
\textbf{Case 2:} if $p=1$, take $\eta\in \mathcal{C}^1(\mathcal{A})(I^n)$ with $d^1 \eta=0$. For any $t\in [a_1,b_1]$, $\eta\big|_{\pi^{-1}(t)}$ is exact by induction hypothesis. We can similarly get a partition of $[a_1,b_1]$ by closed intervals, say $a_1=t_0<t_1<\cdots<t_j<\cdots<t_{M-1}<t_M=b_1$ such that
$$\eta\big|_{\pi^{-1}([t_{j-1},t_j])}=d^0\alpha_{j}$$
for some $\alpha_{j}\in\mathcal{C}^0(\mathcal{A})\big(\pi^{-1}([t_{j-1},t_j])\big)$ where $1\leqslant j\leqslant M$.\\
We can assume that $\alpha_{j}\in\mathcal{C}^0(\mathcal{A})(I^n)$ by the flabbiness of $\mathcal{C}^0(\mathcal{A})$.\\
On $\pi^{-1}(t_{j})$, we have
$$d^{0}\Big((\alpha_{j}-\alpha_{j-1})\big|_{\pi^{-1}(t_{j-1})}\Big)=\eta\big|_{\pi^{-1}(t_{j-1})}-\eta\big|_{\pi^{-1}(t_{j-1})}=0.$$
Therefore, if we identify the constant sheaf $\mathcal{A}$ with its image in $\mathcal{C}^0(\mathcal{A})$,
$$\alpha_{j}-\alpha_{j-1}\in \mathcal{A}(\pi^{-1}(t_{j-1}))$$
is defined on a neighborhood $W_{j-1}$ of $\pi^{-1}(t_{j-1})$.\\
Assume that on $W_{j-1}$
$$\alpha_{j}-\alpha_{j-1}=b_j\in A$$
and define
$$\alpha'=\begin{cases}
\alpha_{j-1}+b_j &\text{\quad on\ }\pi^{-1}([t_{j-2},t_{j-1}])\\
\alpha_j &\text{\quad on\ }\pi^{-1}([t_{j-1},t_{j}])
\end{cases}
$$
Then we get an $\alpha'\in \mathcal{C}^0(\mathcal{A})(\pi^{-1}([t_{j-2},t_j]))$ with $d^0\alpha'=\eta\big|_{\pi^{-1}([t_{j-2},t_j])}$. By making modification for $M-1$ times, we get $\alpha\in \mathcal{C}^{0}(\mathcal{A})(I^n)$ with $d^{0}\alpha=\eta$ on $I^n$. This proves that $H^1(I^n,\mathcal{A})=0$.\\
By induction, we know that $H^p(I^n,\mathcal{A})=0$ for any $n$ and $p> 0$.
\end{proof}
In particular, we know that on closed cubes, $H^p(I^n,\underline{\Z})=H^p(I^n,\underline{\R})=H^p(I^n,\underline{\C})=0$ when $p>0$.
\bibliographystyle{alpha}
\bibliography{main}

\begin{thebibliography}{God73}

\bibitem[Car52]{cartan}
H.~Cartan.
\newblock {\em S\'eminaire des {H}enri {C}artan: Fonctions analytiques de
  plusieurs variables complexes}, volume~4.
\newblock Numdam, 1951--1952.

\bibitem[God73]{Godement}
R.~Godement.
\newblock {\em Topologie algebriqu\'e et th\'eorie des faisceaux}.
\newblock Hermann, Paris, 1973.

\end{thebibliography}
\end{document}